\documentclass[10pt,bezier]{article}
\usepackage{amsmath,amssymb,amsfonts,graphicx,caption,subcaption}
\usepackage[colorlinks,linkcolor=red,citecolor=blue]{hyperref}

\newtheorem{prethm}{{\bf Theorem}}[section]
\newenvironment{thm}{\begin{prethm}{\hspace{-0.5
em}{\bf.}}}{\end{prethm}}
\newtheorem{prepro}{{\bf Theorem}}

\newtheorem{precor}[prethm]{{\bf Corollary}}
\newenvironment{cor}{\begin{precor}{\hspace{-0.5
em}{\bf.}}}{\end{precor}}
\newtheorem{preconj}[prethm]{{\bf Conjecture}}
\newenvironment{conj}{\begin{preconj}{\hspace{-0.5
em}{\bf.}}}{\end{preconj}}
\newtheorem{preremark}[prethm]{{\bf Remark}}
\newenvironment{remark}{\begin{preremark}\em{\hspace{-0.5
em}{\bf.}}}{\end{preremark}}
\newtheorem{prelem}[prethm]{{\bf Lemma}}

\newtheorem{preque}[prethm]{{\bf Question}}

\newtheorem{preobserv}[prethm]{{\bf Observation}}

\newtheorem{predef}[prethm]{{\bf Definition}}

\newtheorem{preproposition}[prethm]{{\bf Proposition}}
\newenvironment{proposition}{\begin{preproposition}{\hspace{-0.5
em}{\bf.}}}{\end{preproposition}}

\newtheorem{preproof}{{\bf Proof.}}
\newtheorem{preprooff}{{\bf Proof}}

\newenvironment{proof}[1]{\begin{preproof}{\rm
#1}\hfill{$\Box$}}{\end{preproof}}

\newtheorem{preproofs}{{\bf The second proof of }}

\newtheorem{preprooft}{{\bf Third proof of }}

\newtheorem{preproofF}{{\bf Proof of}}

\title{\bf\Large 
Modulo orientations with bounded out-degrees
}
\author{{\normalsize{\sc Morteza Hasanvand${}$} }\vspace{3mm}
\\{\footnotesize{${}$\it Department of Mathematical
 Sciences, Sharif
University of Technology, Tehran, Iran}}
{\footnotesize{}}\\{\footnotesize{ $\mathsf{morteza.hasanvand@alum.sharif.edu }$ }}}

\date{}

\begin{document}
\maketitle
\begin{abstract}{
Let $G$ be a graph, let $k$ be a positive integer, and let $p:V(G)\rightarrow Z_k$ be a mapping with $|E(G)| \stackrel{k}{\equiv}\sum_{v\in V(G)}p(v) $. In this paper, we show that if $G$ is $(3k-3)$-edge-connected, then it has an orientation such that for each vertex $v$, $|d^+_G(v)-d_G(v)/2| < k$; also if $G$ contains $2k-2$ edge-disjoint spanning trees, then it admits such an orientation but by imposing greater out-degree bounds.
} {\small
}
\end{abstract}
%
%==============================================================================
%
%
%
%
%									 1
%								   111
%									11
%									11
%									11
%									11
%								  1111111
%
%
%
%
%==============================================================================
%
\section{Introduction}
In this article, graphs have no loops, but multiple edges are allowed, and a general graph 
may have loops and multiple edges.
Let $G$ be a graph. The vertex set and the edge set of $G$ are denoted by $V(G)$ and $E(G)$, respectively.
We denote by $d_G(v)$ the degree of a vertex $v$ in the graph $G$, whether $G$ is directed or not.
If $G$ has an orientation, the out-degree and in-degree of $v$ are denoted by $d_G^+(v)$ and $d_G^-(v)$.
For a vertex set $A$ of $G$ with at least two vertices, the number of edges of $G$ with exactly one end in $A$ is denoted by $d_G(A)$.
Also, we denote by $e_G(A)$ the number of edges with both ends in $A$ 
and denote by $d_G(A,B)$ the number of edges with one end in $A$ and one end in $B$, where $B$ is a vertex set.
For notational
simplicity, we write $A^c$ for the vertex set $V(G)\setminus A$.
We denote by $G[A]$ the induced subgraph of $G$ with the vertex set $A$ containing
precisely those edges of $G$ whose ends lie in $A$, and denote by $G[A, B]$ the induced bipartite factor of $G$
with the bipartition $(A, B)$.
The graph
$G$ is said to be {\bf trivial}, if it has no edges.
Let $k$ be a positive integer. 
The cyclic group of order $k$ is denoted by $Z_k$.
For any integer $n$, we denote by $[n]_k^-$ the unique integer $n_0$ such that $n_0\stackrel{k}{\equiv}n$ and 
$n_0\in \{-1,0,\ldots, k-2\}$.
For convenience, we write $[n]_k$ for $[n]^-_k$.
An orientation of $G$ is said to be 
{\bf $p$-orientation}, if for each vertex $v$, $d_G^+(v)\stackrel{k}{\equiv}p(v) $,
 where $p:V(G)\rightarrow Z_k$ is a mapping.
For two rational numbers $x$ and $y$,
 we say that $x\stackrel{k}{\equiv}y$, if $x-y$ is an integer divisible by $k$. 
A graph $G$ is called {\bf $m$-tree-connected}, if it contains $m$ edge-disjoint spanning trees. 
Note that by the result of Nash-Williams~\cite{Nash-Williams-1961} and Tutte~\cite{Tutte-1961} every $2m$-edge-connected graph is $m$-tree-connected.
A graph $G$ is said to be {\bf $(m,l_0)$-partition-connected},
 if it can be decomposed into an $m$-tree-connected factor and a factor $F$ having an orientation such that for each vertex $v$, $d^+_F(v)\ge l_0(v)$,
where $l_0$ is an integer-valued function on $V(G)$. 
A graph is termed {\bf essentially $\lambda$-edge-connected},
 if all edges of any edge cut of size strictly less than $\lambda$ are incident with a common vertex.
A graph $G$ is called {\bf odd-$\lambda$-edge-connected}, if $d_G(A)\ge \lambda$ for every vertex set $A$ with $d_G(A)$ odd.
For a graph $G$ with a given vertex $u$, we denote by $\chi_u$ the mapping $\chi_u:V(G)\rightarrow \{0,1\}$
such that $\chi_u(u)=1$ and $\chi_u(v)=0$ for all vertices $v$ with $v\neq u$.
Two different edges are called {\bf parallel}, if have the same end vertices.
 For two edges $xu$ and $uy$ incident with the vertex $u$, {\bf lifting of $xu$ and $uy$}
 is an operation that removes $xu$ and $uy$ and adds a new edge $xy$
(when the purpose is to generate a loopless graph we must not add the next edge $xy$ when $x=y$).
Also, if one of $xu$ and $yu$, as $xu$, is directed, then we direct $xy$ toward $y$ when $xu$ is
toward $u$, and direct $xy$ away from $y$ when $xu$ is
 away from $u$.
Throughout this article, all variables $m$ are nonnegative integers and all variables $k$ are positive integers.
%

%==============================================================================
In 2012 Thomassen constructed the following theorem about the existence of modulo orientations in highly edge-connected graphs.
\begin{thm}{\rm(\cite{MR2885433})}\label{Intro:orientation:thm:Thomassen}
{Let $G$ be a graph, let $k$ be an integer, $k\ge 3$, and let $p:V(G)\rightarrow Z_k$ be a mapping with
$|E(G)| \stackrel{k}{\equiv} \sum_{v\in V(G)}p(v)$.
 If $G$ is $(2k^2+k)$-edge-connected, then
it has a $p$-orientation.
}\end{thm}

Later, Lov{\'a}sz, Thomassen, Wu, and Zhang (2013) refined 
Theorem~\ref{Intro:orientation:thm:Thomassen} for odd integers $k$ by reducing the quadratic bound $(2k^2 + k)$ down to a linear bound as the following theorem. They also remarked that this number can be reduced to $3k-2$ for even integers $k$.
\begin{thm}{\rm(\cite{Lovasz-Thomassen-Wu-Zhang-2013})}\label{Into:thm:3k-2}
{Let $G$ be a graph, let $k$ be an odd integer, $k\ge 3$, and let $p:V(G)\rightarrow Z_k$ be a mapping with
$|E(G)| \stackrel{k}{\equiv} \sum_{v\in V(G)}p(v)$.
If $G$ is $(3k-3)$-edge-connected, then $G$ has a $p$-orientation.
In particular, the needed edge-connectivity can be replaced by odd-edge-connectivity $3k-2$, when $2p(v)\stackrel{k}{\equiv}d_G(v)$ for all vertices~$v$.
}\end{thm}

In this paper, we refine Theorem~\ref{Into:thm:3k-2} by pushing the required edge-connectivity down to $3k-3$, even for even numbers $k$, and strengthen it by giving a sharp bound on out-degrees as mentioned in the abstract. In addition, we improve the needed edge-connectivity for Eulerian graphs for the following special case.
\begin{thm}
{Let $k$ be an even positive integer, let $G$ be an Eulerian graph, and let $Q\subseteq V(G)$ with $|Q|$ even.
If $d_G(X)\ge 3k-2$ for every $X\subseteq V(G)$ with $|X\cap Q|$ odd,
then $G$ admits an orientation such that for each $v\in V(G)\setminus Q$, $d_G^+(v)=d_G^-(v)$, and 
for each $v\in Q$, $|d_G^+(v)-d_G^-(v)|=k$.
}\end{thm}

Recently, Thomassen (2020) applied Theorems~\ref{Intro:orientation:thm:Thomassen}
 and~\ref{Into:thm:3k-2} to establish the following elegant result about the existence of regular factorizations of edge-connected regular graphs. As an application, we applied these new improvements in~\cite{Equitable} to refine Thomassen's result.
\begin{thm}{\rm (\cite{Thomassen-2020})}
{Let $r$ be a natural number, $r\ge 4$ and let $G$ be an $r$-regular graph. Let $r = kq$ be 
a natural number where $q \ge 3$ is odd. If $k$ is odd and $G$ has odd-edge-connectivity at least 
$3k - 2$, then $G$ can be edge-decomposed into $q$-factors. If $k$ is even and $G$ has an even 
number of vertices and edge-connectivity at least $k^2 +2k$, then $G$ can be edge-decomposed 
into $q$-factors.
}\end{thm}

In Section~\ref{sec:partition-connected}, we investigate modulo orientations with bounded out-degrees in partition-connected graphs and strengthen some recent results in~\cite{Han-Lai-Li-2018, ManyEdgeDisjoint} toward this concept. 
As a consequence, we prove the following theorem.
\begin{thm}
{Let $G$ be a graph with $z\in V(G)$, let $k$ be an integer, $k\ge 3$, and let $p:V(G)\rightarrow Z_k$ be a mapping with
$|E(G)| \stackrel{k}{\equiv} \sum_{v\in V(G)}p(v)$.
If $G$ is $(2k-2)$-tree-connected, then it has a $p$-orientation such that for each $v\in V(G)\setminus \{z\}$, 
$$k-1\le d^+_G(v) \le \max\{2k-2, d_G(v)-(k-1)\}.$$
}\end{thm}
%
%
%==============================================================================
%
%
%
%
%
%
%
%
%
%
%
%
%==============================================================================
%
\section{Orientations modulo $2$}
In this section, we consider the existence of parity orientations. 
Our results are based on the following theorem which is a special case of a result due 
to Frank, Tardos, and Seb{\H{o}} (1984) who gave a criterion for the existence of parity factors with bounded degrees.
\begin{thm}{\rm(see Theorem 6 in~\cite{Frank-Tardos-Sebo-1984})}\label{thm:FTS:parity}
{Let $G$ be a connected graph and let $p$ and $q$ be two integer-valued functions on $V(G)$ with $p\le q$ satisfying $|E(G)|\stackrel{2}{\equiv}\sum_{v\in V(G)}p(v)$ and $p(v)\stackrel{2}{\equiv}q(v)$ for each vertex $v$.
Then $G$ has a $p$-orientation modulo $2$ such that for each vertex $v$, $p(v) \le d^+_G(v)\le q(v)$,
if for any two disjoint subsets $A$ and $B$ of $V(G)$ with $A\cup B\neq \emptyset$,
$$\omega(G\setminus (A\cup B))\le 1+ \sum_{v\in A}q(v)-e_G(A)+ \sum_{v\in B}(d_{G}(v)-p(v))-e_G(B).$$
}\end{thm}
\begin{cor}{\rm(see Theorem 4 in~\cite{Frank-Tardos-Sebo-1984})}\label{Orientation:modulo2:cor}
{Let $G$ be a graph and let $p:V(G)\rightarrow Z_2$ be a mapping with
$|E(G)|\stackrel{2}{\equiv}\sum_{v\in V(G)}p(v)$.
 If $G$ is connected, then $G$ has a $p$-orientation.
}\end{cor}
%
%==============================================================================
%
%
%
%
%
%==============================================================================
%
\subsection{$2$-edge-connected graphs}
As we have stated above, edge-connectedness $1$ is sufficient for a graph to have a $p$-orientation modulo~$2$.
Here, we show that edge-connectedness $2$ is sufficient for a graph to have a $p$-orientation modulo~$2$
 in which out-degrees fall in predetermined short intervals.
\begin{thm}\label{Orientation:modulo2:thm}
{Let $G$ be a graph and let $p:V(G)\rightarrow Z_2$ be a mapping with
$|E(G)|\stackrel{2}{\equiv}\sum_{v\in V(G)}p(v)$.
 If $G$ is $2$-edge-connected, then it has a $p$-orientation such that for each vertex $v$, 
$$ \lfloor \frac{d_G(v)}{2}\rfloor-1
\le d_G^+(v)\le 
\lceil \frac{d_G(v)}{2}\rceil +1.$$
Furthermore, for an arbitrary vertex $z$, $d^+_{G}(z)$ 
can be assigned to any plausible integer value in whose interval.
}\end{thm}
\begin{proof}
{For each vertex $v$, define $p'(v)\in \{\lfloor d_{G}(v)/2\rfloor-1, \lfloor d_{G}(v)/2\rfloor\}$ and
 $q'(v)\in \{ \lceil d_{G}(v)/2\rceil, \lceil d_{G}(v)/2\rceil+1\}$ 
 such that 
$p'(v)\stackrel{2}{\equiv}p(v)\stackrel{2}{\equiv}q'(v)$. 
Obviously, $p'(v)\le q'(v)$. 
Note that if $p'(z)<q'(z)$, we permit to replace $p'(z)$ by $p'(z)+2$ or 
 replace $q'(z)$ by $q'(z)-2$ with respect to our purpose related to $z$.
For the first option, we have $p'(z)\le d_G(z)/2+3/2$ and for the second option, we have $d_G(z)/2\le q'(z)+3/2$.
Let $A$ and $B$ be two disjoint subsets of $V(G)$ with $A\cup B\neq \emptyset$. Since $G$ is $2$-edge-connected, it is easy to see that
$$\omega(G\setminus (A\cup B)) \le \sum_{A\cup B}\frac{1}{2}d_G(v)-e_G(A\cup B),$$
which implies that 
$$\omega(G\setminus (A\cup B))\le 3/2+\sum_{v\in A}q'(v)-e_G(A)+\sum_{v\in B}(d_G(v)-p'(v)) -e_G(B).$$
Thus by Theorem~\ref{thm:FTS:parity}, the graph $G$ has a $p$-orientation such that for each vertex $v$, $p'(v) \le d^+_G(v)\le q'(v)$, and the proof is completed.
}\end{proof}
We shall here introduce an alternative proof using an induction.
\begin{proof}
{By induction on the sum of all $d_G(v)-3$ taken over all vertices $v$ with $d_G(v)\ge 4$. 
First, assume that for each vertex $v$, $d_G(v)\le 3$. 
For $|V(G)|=1$, the proof is clear. 
So, suppose $|V(G)|\ge 2$ which implies that $G$ has no loops.
It is not hard to check that there is an edge set $E$ incident with $z$ such that $G-E$ is connected, 
where $|E|=1$ when $z$ has even degree and $|E|=2$ when $z$ has odd degree.
Orient the edge(s) of $E$ toward $z$, if the desired condition on out-degree of $z$ is $d^+_G(z)\le\lfloor d_G(z)/2\rfloor$,
and orient the edge(s) of $E$ away from $z$
 if the desired condition on out-degree of $z$ is $ d^+_G(z)\ge \lceil d_G(z)/2\rceil $.
By applying Corollary~\ref{Orientation:modulo2:cor} to the graph $G-E$, the pre-orientation of $E$ can be extended
 to a $p$-orientation of $G$ 
 satisfying 
the theorem.
Now, assume that for a vertex $u$, $d_G(u) \ge 4$. 
By Fleischner's splitting lemma,
there are two edges $xu$ and $yu$ of $G$ incident with $u$ such that by lifting them the resulting general graph $H$ is still $2$-edge-connected (possibly $x=u$ or $y=u$).
Note that for each vertex $v$, $d_H(v)=d_G(v)-2\chi_u$.
Define $p'=p-\chi_{u}$ so that 
$|E(H)|\stackrel{2}{\equiv}\sum_{v\in V(H)}p'(v)$.
By the induction hypothesis, $H$ admits a $p'$-orientation such that for each vertex $v$,
$ \lfloor d_H(v)/2\rfloor-1
\le d_H^+(v)\le 
\lceil d_H(v)/2\rceil +1$.
This orientation of $H$ induces an orientation for $G$ such that for each vertex $v$, $d_G^+(v)=d^+_H(v)+\chi_u$.
This orientation of $G$ is a $p$-orientation
 satisfying
 the desired properties.
The extra condition on $d^+_{G}(z)$ can be obtained by giving an appropriate condition on $d^+_{H}(z)$. 
Hence the theorem holds.
}\end{proof}
%
%
%
%==============================================================================
%
%
%
%
%
%
%==============================================================================
%
%
\subsection{$(1,l_0)$-partition-connected graphs}
In the following theorem, 
we develop Theorem~\ref{Orientation:modulo2:thm} to a partition-connected version.
\begin{thm}\label{mixed:modulo2:thm:2:partition}
{Let $G$ be a graph and let $p:V(G)\rightarrow Z_2$ be a mapping with $|E(G)| \stackrel{2}{\equiv} \sum_{v\in V(G)}p(v) $. 
Let $s$, $s_0$, and $l_0$ be three integer-valued functions on $V(G)$ satisfying $s+s_0< d_G$ and
$\max\{s,s_0\}\le l_0$. 
If $G$ is $(1, l_0)$-partition-connected, then it has a $p$-orientation such that for each vertex~$v$,
$$s(v)\le d_{G}^+(v)\le d_G(v)-s_0(v).$$
}\end{thm}
\begin{proof}
{For each vertex $v$, define $p'(v)\in \{s(v), s(v)+1\}$ and $q'(v)\in \{d_G(v)-s_0(v)-1, d_G(v)-s_0(v)\}$ such that 
$p'(v)\stackrel{2}{\equiv}p(v)\stackrel{2}{\equiv}q'(v)$. 
Note that the condition $s(v)+s_0(v)< d_G(v)$ implies that $p'(v)\le q'(v)+1$ and hence $p'(v)\le q'(v)$. 
By the assumption, the graph $G$ can be decomposed into two factors $T$ and $F$ such that $T$ is a spanning tree and $F$ admits an orientation such that for each vertex $v$, $d^+_F(v)\ge l_0(v)$.
Let $A$ and $B$ be two disjoint subsets of $V(G)$. Since $T$ is connected, it is easy to see that
$$\omega(G\setminus (A\cup B)) \le \omega(T\setminus (A\cup B))\le \sum_{A\cup B}(d_T(v)-1)+1-e_T(A\cup B).$$
Moreover, since $d^+_F(v)\ge l_0(v)\ge \max \{s(v),s_0(v)\}$ for each vertex $v$, we must have 
$$0\le \sum_{v\in A\cup B}d^-_F(v)-e_F(A\cup B)\le 
\sum_{v\in A}(d_F(v)-s_0(v))+\sum_{v\in B}(d_F(v)-s(v))-e_F(A\cup B).$$
 Therefore,
$$\omega(G\setminus (A\cup B))\le \sum_{v\in A}(d_G(v)-s_0(v)-1)+
\sum_{v\in B}(d_G(v)-s(v)-1) -e_{G}(A\cup B)+1,$$
which implies that 
$$\omega(G\setminus (A\cup B))\le 1+\sum_{v\in A}q'(v)-e_{G}(A)+
\sum_{v\in B}(d_G(v)-p'(v)) -e_{G}(B).$$
Thus by Theorem~\ref{thm:FTS:parity}, the graph $G$ has a $p$-orientation such that for each vertex $v$, $p'(v) \le d^+_G(v)\le q'(v)$, and the proof is completed.
}\end{proof}
%
%==============================================================================
%
%
%
%
%
%
%
%
%
%
%
%
%==============================================================================
%
\section{Orientations modulo $k$: edge-connected graphs}
In this section, we are going to improve Theorem~\ref{Into:thm:3k-2} by giving a sharp bound on out-degrees 
and provide a common version for odd and even integers $k$.
We follow with the same innovative ideas that appeared in~\cite{Lovasz-Thomassen-Wu-Zhang-2013}
and retain the same arguments, while modifications are inserted.
The proof is based on defining a set function $\alpha$ whose values lie in the set $\{0,\pm 1/2,\ldots,\pm k/2\}$. It
 is inspired by the set function $\tau(A)$ in~\cite{Lovasz-Thomassen-Wu-Zhang-2013} and the set function $t(A)$ in~\cite{MR2885433}.
More precisely, for odd integers $k$, $2\alpha(A)=\tau(A)$, and for odd and even integers $k$, $2|\alpha(A)|=t(A)$.
\subsection{Definition and properties of set functions $\alpha$}

Let $G$ be a
%loopless 
graph, let $k$ be a positive integer, and let $p:V(G)\rightarrow Z_k$ be a mapping.
For each vertex $v$, take $\alpha(v)$ to be a rational number such that
$\alpha(v)\in \{0,\pm 1/2,\ldots,\pm k/2\}$ and
$\alpha(v)\stackrel{k}{\equiv}p(v)-d_G(v)/2$.
In intuitive terms, $|\alpha(v)|$ specifies the distance between two points $p(v)$ and $d_G(v)/2$ on a circle whose circumference is $k$,
 and the sign of $\alpha(v)$ determines the position of $p(v)$ with respect to $d_G(v)/2$.
Thus, it is intuitively clear and not difficult to show that $\alpha(v)$ is unique unless $\alpha(v)\in \{-k/2,k/2\}$.
For any vertex set $A$, take $\alpha(A)$ to be a rational number such that
$\alpha(A)\in \{0,\pm 1/2,\ldots,\pm k/2\}$ and 
$\alpha(A)\stackrel{k}{\equiv}p(A)-d_G(A)/2,$
where $p(A)=\sum_{v\in A} p(v)\,-e_G(A)$ and $d_G(A)=\sum_{v\in A} d_G(v)\, -2e_G(A)$.
When $G$ and $p$ are not clear from the context, 
we denote by $\alpha_G(v,p)$ and $\alpha_G(A,p)$ the value of $\alpha(v)$ and $\alpha(A)$.
Now, we present some basic properties of $\alpha$ in the following propositions.
\begin{proposition}\label{Orientation:3k-3:proposition}
{Let $G$ be a graph and let $p:V(G)\rightarrow Z_k$ be a mapping with
$|E(G)|\stackrel{k}{\equiv}\sum_{v\in V(G)}p(v)$.
For any two vertex sets $A$ and $B$, the following statements hold:
\begin{enumerate}{
\item \label{con:1} 
If $\alpha(A)\stackrel{k}{\equiv}\pm \alpha(B)$, then $|\alpha(A)|=|\alpha(B)|$.
\item \label{con:2} 
If $A\cap B=\emptyset$, then
$\alpha(A\cup B)\stackrel{k}{\equiv}\alpha(A)+\alpha(B).$
\item \label{con:3} 
$|\alpha(A)|=|\alpha(A^c)|.$
\item
\label{con:4} 
If $\alpha(v_0)=0$ for a vertex $v_0$ with $v_0\in V(G)\setminus A$, then
$|\alpha(A)|=|\alpha(A\cup \{v_0\})|.$
\item \label{con:5} 
If $d_G(A)\ge 3k-3$, then $d_G(A)\ge (2k-2)+2|\alpha(A)|$.
\item \label{con:6}  
$d_G(A)-2|\alpha(A)|$ is an even integer.
}\end{enumerate}
}\end{proposition}
\begin{proof}
{To obtain (\ref{con:1}), one can conclude that $|\alpha(A)\mp \alpha(B)|\in \{0,k\}$ which implies that $|\alpha(A)|=|\alpha(B)|$.
To prove (\ref{con:2}), it suffices to check that
$\alpha(A)+\alpha(B) \stackrel{k}{\equiv} (p(A)-d_G(A)/2)+(p(B)-d_G(B)/2)
 \stackrel{k}{\equiv} \big(p(A)+p(B)-d_G(A,B)\big) -\big(d_G(A) +d_G(B)-2d_G(A, B)\big)/2$
which implies that
$ \alpha(A)+\alpha(B) \stackrel{k}{\equiv} p(A\cup B) -d_G(A\cup B)/2
 \stackrel{k}{\equiv} \alpha(A \cup B).$
%
%where $e_G(A,B)$ denotes the number of edges $ab$ with $a\in A$ and $b$. 
Moreover, $\alpha\big(V(G)\big)=0$, since 
$p\big(V(G)\big) \stackrel{k}{\equiv}\sum_{v\in V(G)}p(v)\; - e_{G}(V(G)) \stackrel{k}{\equiv}\sum_{v\in V(G)}p(v)\;- |E(G)| \stackrel{k}{\equiv}0.$
Hence 
$\alpha(A)+\alpha (A^c)\stackrel{k}{\equiv} 0$
 and 
$|\alpha(A)|=|\alpha (A^c)|$ which establishes (\ref{con:3}). 
The proof of~(\ref{con:4}) can be obtained from
$\alpha(A) \stackrel{k}{\equiv} \alpha(A)+ \alpha(v_0)\stackrel{k}{\equiv}\alpha(A\cup \{v_0\})$.
Since $\alpha(A)+d_G(A)/2$ is an integer, 
$2\alpha(A)+d_G(A)$ is even and so $d_G(A)-2|\alpha(A)|$ is even which implies (\ref{con:6}).
Note that $|\alpha(A)|\le k/2$.
If $|\alpha(A)|=k/2$, then $d_G(A)$ and $k$ have the same parity.
Since $3k-3$ and $k$ have different parity, we have $|\alpha(A)|<k/2$, when $d_G(A)=3k-3$.
 This can complete the proof.
}\end{proof}
\begin{proposition}
{Let $G$ be a graph and let $p:V(G)\rightarrow Z_k$ be a mapping with
$|E(G)|\stackrel{k}{\equiv}\sum_{v\in V(G)}p(v)$.
If $G'$ is a graph obtained from $G$ by lifting two edges $xu$ and $yu$, then for every vertex set $A$ we have $|\alpha'(A)|=|\alpha(A)|$, where 
$$p' =
 \begin{cases}
p-\chi_u-\chi_x, &\text{if $xu$ and $yu$ are parallel};\\
p-\chi_u, &\text{if $xu$ and $yu$ are not parallel}.
 \end{cases}$$
}\end{proposition}
%
%
%
%
%
%====================================
%

\subsection{Graphs with edge-connectivity at least $3k-3$}
Now, we are ready to refine the main result in~\cite{Lovasz-Thomassen-Wu-Zhang-2013}. 
\begin{thm}\label{Orientation:3k-3:thm:alpha}
{Let $G$ be a
graph with $z_0\in V(G)$, let $k$ be an integer, $k\ge 3$, and let $p:V(G)\rightarrow Z_k$ be a mapping with
$|E(G)|\stackrel{k}{\equiv} \sum_{v\in V(G)}p(v)$. Let $D_{z_0}$ 
be a pre-orientation 
of $E(z_0)$ that is the set of edges incident with $z_0$.
Let $V_0=\{v\in V(G)-z_0: \alpha(v)=0\}$.
 If $V_0\neq \emptyset$, we let $v_0$ be a vertex of $V_0$ with smallest degree. Assume that 
\begin{enumerate}{
%\item [{\upshape$(i)$}]  %1
%$|V(G)|\ge 2$.

\item [{\upshape$(i)$}]  %1
$d_G(z_0)\le 2k-2+2|\alpha(z_0)|$, and the edges incident with $z_0$ are pre-directed such that $d^+_G(z_0)\stackrel{k}{\equiv}p(z_0)$.

\item [{\upshape$(ii)$}]  %1
$d_G(A)\ge 2k-2+2|\alpha(A)|$, for any vertex set $A$ with $\emptyset \subsetneq A \subsetneq V(G)\setminus \{z_0\}$ and $A\neq \{v_0\}$.
}\end{enumerate}

Then the pre-orientation $D_{z_0}$ can be extended to a $p$-orientation $D$ of $G$ such that for each vertex~$v$,
$$|d^+_G(v)-d_G(v)/2| \le k-1+|\alpha(v)|.
$$
}\end{thm}
\begin{proof}
{The proof is by contradiction. We assume (reductio ad absurdum) that $(G, p, z_0)$ is a counterexample
so that $|V(G)|\ge 3$. 
That is, the
graph $G$ with mapping 
$p$ satisfies
 the conditions of the theorem but some
pre-orientation $D_{z_0}$ cannot be extended to a $p$-orientation of $G$
 with the desired properties.
Let $\mathcal{M}$ be the collection of counterexamples $(G, p, z_0)$ such that $|V (G)|+|E(G - z_0)|$ is minimum.
The proof is divided into two parts.
The first part, Claims 1-5 below, establishes some properties
of all members of $\mathcal{M}$.
 In the second part we choose a member $(G, p, z_0)$ of $\mathcal{M}$ such that $|E(G)|$ is
minimum and prove that it is not a counterexample, yielding a contradiction.
If we work with distinct
 graphs $G$, $G'$,
we use the terms $p(A)$ and $\alpha (A)$ when $A$ is a vertex
set of $G$, and $p'(A)$ and $\alpha'(A)$ when $A$ is a vertex set of~$G'$.
\vspace{5mm}
\\
%--------------------------------------------------------------------
{\bf Part I. Some properties of $\mathcal{M}$.}

In Part 1 we let $(G,p, z_0)$ be any member of $\mathcal{M}$.
\vspace{3mm}
\\
{\bf Claim 1. For every vertex set $A\subsetneq V(G)\setminus \{z_0\}$ with $|A|\ge 2$, we have $d_G(A)\ge 2k+2|\alpha(A)|$.}

If $d_G(A)< 2k+2|\alpha (A)|$, then we first get an extension of $D_{z_0}$ to the contracted graph $H = G/A$ by the
minimality property of $G$, since $|V (H)| < |V (G)|$ and $|E(H -z_0)| \le |E(G - z_0)|$. 
Then all edges of the
edge-cut $[A, A^c]$ are oriented in this extension, where $A^c = V (G)\setminus A$ and $[A, A^c]$ is the set of edges with exactly one end in $A$.
Similarly we then contract $A^c$
into a single vertex as a new $z_0$, and again, we use the minimality of $G$ to extend the orientation of
$[A, A^c]$ to the edges of $G$ with both ends in $A$. $\square$
\vspace{3mm}
\\
{\bf Claim 2. $V_0= \emptyset$.}

Suppose $V_0\neq \emptyset$ and $v_0$ is a vertex of $V_0$ with smallest degree.
We can assume that $d_G(v_0) \ge 2$. 
Otherwise $v_0$ is an isolated vertex and we can remove it and use the minimality of $G$.
If $v_0$ has at least two neighbours, we lift one pair of edges incident with $v_0$
 which are not parallel. 
Claim 1 implies that
the resulting graph $G'$ with the modified mapping $p'=p-\chi_{v_0}$
 satisfies the hypotheses of the theorem. Since 
$|E(G'- z_0)| < |E(G - z_0)|$, it holds
that $G'$ has the desired orientation, and 
so
 does $G$, a contradiction.

Now suppose $v_0$ has only one neighbour $x$.
We must have $x \neq z_0$. Otherwise, 
$$|\alpha(W )| =|\alpha (\{z_0, v_0\})| =|\alpha (z_0)|,$$
where 
$W = V (G) \setminus \{z_0,v_0\}$,
 and
then 
$$d_G(z_0) = d_G(W ) + d_G(v_0)\ge (2k - 2) +2|\alpha (W )|+ 2 = 2k +2|\alpha (z_0)|.$$
a contradiction to 
condition (i). 
If $|V (G)|= 3$, then we extend $D_{z_0}$ to an orientation of $G$ by
orienting half of the edges between $x$ and $v_0$ toward $v_0$ and the other half away from $v_0$, yielding a
contradiction.
For the case $|V (G)|> 3$, we have
$$d_G(x) = d_G(\{x,v_0\} ) + d_G(v_0)\ge (2k - 2) +2|\alpha (\{x,v_0\} )|+ 2 = 2k +2|\alpha (x)|.$$
Then, we lift one pair of edges incident with $v_0$ and $x$ which are parallel.
Claim~1 implies that
the resulting graph $G'$ with the modified mapping $p'=p-\chi_{x}-\chi_{v_0}$ 
 satisfies the hypotheses of the theorem. 
Since 
$|E(G'- z_0)| < |E(G - z_0)|$, it holds
that $G'$ has the desired orientation, and 
so
 does $G$, again a contradiction. 
$\square$
%--------------------------------------------------------------------
%
\vspace{3mm}
\\
{\bf Claim 3. $G-z_0$ is connected, and $d_G(z_0) \ge k$.}

Suppose $G - z_0$ is disconnected and let $U$ and $W$ be two components of $G - z_0$. 
By 
condition~(ii)
and Claim 2, we have $d_G(U) \ge 2k - 2$ and $d_G(W ) \ge 2k - 2$. Then
$$d_G(z_0) \ge d_G(U) + d_G(W )>(2k - 2) + k \ge (2k - 2) +2|\alpha(z_0)|,$$
a contradiction to
 condition~(i).

Suppose $d_G(z_0) \le k - 1$ and let $G'$
 be the graph constructed from $G$ by replacing an edge $xy$ of
$G - z_0$ with a directed path of length two through $z_0$ 
 with $p' = p+\chi_{z_0}$.
We have $d_{G'}(z_0) \le (k - 1) + 2 \le 2k - 2 \le (2k - 2) +2|\alpha' (z_0)|$ and hence $G'$
 satisfies 
condition (i). 
For any vertex set $A$ described in
condition (ii), 
$d_{G'}(A) = d_G(A) + 2$, if $A$ contains both $x$ and $y$, and $d_{G'}(A) = d_G(A)$ otherwise. 
So condition (ii) is clearly satisfied.
 Since $|V (G')|=|V (G)|$ and $|E(G'- z_0)| < |E(G - z_0)|$, this implies (by the definition of $\mathcal{M}$) that an extension of $D_{z_0}$ exists in $G'$. 
This orientation results in an orientation of $G$,
a contradiction.
$\square$
%--------------------------------------------------------------------
%
\vspace{3mm}
\\
{\bf Claim 3.A.
For each vertex $v\in V(G) - z_0$, $d_G(v)= 2k-2+2|\alpha(v)|$.}

Suppose otherwise that $d_G(x)\ge 2k+2|\alpha(x)|$ for a vertex $x$ with $x\neq z_0$.
First, assume that $x$ has at least two neighbours. 
Then we lift one pair of edges incident with $x$
which are not parallel.
Claim~1 implies that
the resulting graph $G'$ with the modified mapping $p'=p-\chi_{x}$ 
 satisfies the hypotheses of the theorem. 
Since 
$|E(G'- z_0)| < |E(G - z_0)|$, it holds
that $G'$ has the desired orientation, and 
so does $G$, a contradiction.
Next, assume that $x$ has only one neighbour $y$.
By Claim 3, we must have $y\neq z_0$ and so
$$d_G(y)\ge 
 \begin{cases}
d_G(x)+d_G(z_0)\ge 2k+k\ge\; 2k+2|\alpha(y)|, &\text{if $|V(G)|=3$};\\
d_G(x)+d_G( \{x,y\})\ge 2k+2k-2\ge\; 2k+2|\alpha(y)|, &\text{if $|V(G)|>3$}.
 \end{cases}$$

Then, we lift one pair of edges incident with $x$ and $y$
which are parallel.
Claim~1 implies that
the resulting graph $G'$ with the modified mapping $p'=p-\chi_{x}-\chi_y$ 
 satisfies the hypotheses of the theorem. 
Since 
$|E(G'- z_0)| < |E(G - z_0)|$, it holds
that $G'$ has the desired orientation, and 
so
 does $G$, again a contradiction. 
$\square$

By condition (i) and Claim 3.A, if $G$ has a $p$-orientation, then for each vertex $v$
 the following condition automatically holds,
$$|d^+_G(v)-d_G(v)/2| \le k-1+|\alpha(v)|.$$
%--------------------------------------------------------------------
%
%\vspace{3mm}
%\\
{\bf Claim 4. For any two distinct vertices $x, y\in V (G) - z_0$, we have $\alpha (x)\alpha (y)> 0$.}

Suppose $\alpha(x)\alpha (y) \le 0$. By Claim 2, we may assume that $\alpha (x)> 0$ and $\alpha (y)< 0$.
By Claim~3, since $G - z_0$ is connected, we may also assume that $xy\in E(G-z_0)$.
%.
 Let
$G'= G -xy$, and take $p'=p-\chi_{x}$ to be the modified mapping.

Then $|V (G')|=|V (G)|$ and $|E(G'-z_0)| < |E(G - z_0)|$. 
If $G'$ and $p'$ satisfy the conditions of the theorem,
 then by the definition of $\mathcal{M}$,
 the pre-orientation can be extended to a $p'$-orientation of $G'$ 
and further to a $p$-orientation of $G$
by adding a directed edge from 
$x$ to $y$,
 yielding a contradiction.
Hence, it suffices to verify the conditions of the theorem for $G'$ and $p'$.
 Moreover, we only need
to verify 
condition (ii) 
for single vertices $x$ and $y$
 and vertex sets $A$ such that $|A| \ge 2$ and
$d_{G'}(A) = d_G(A) - 1$ which are affected by the deletion of $xy$.

Condition (ii) 
is satisfied for $x $ and $y$, since 
\begin{align*} 
d_{G'}(x)=& d_{G}(x)-1,&
p'(x)&=p(x)-1,& 
\alpha'(x)&=\alpha(x)-1/2, 
& |\alpha'(x)| &=|\alpha(x)|-1/2,\\
d_{G'}(y)=& d_{G}(y)-1,&
p'(y)&=p(y),& 
\alpha'(y)&=\alpha(y)+1/2, 
& |\alpha'(y)| &=|\alpha(y)|-1/2.
\end{align*}
 For any vertex set $A$ 
(in condition (ii)) 
such that 
$|A| \ge 2$ and $d_{G'}(A) = d_G(A) - 1$, 
we have 
$|\alpha' (A)|=|\alpha (A) \pm 1/2| \le |\alpha (A)|+ 1/2$ 
and by Claim 1, 
$$d_{G'}(A) = d_G(A) - 1\ge \big(2k +2|\alpha (A)|\big)- 1 \ge (2k - 2) +2|\alpha' (A)|.$$
Hence 
condition~(ii)
 is verified for $A$. So $\alpha (x)\alpha (y)>0$.
$\square$

Let $V^+ = \{x \in V (G) - z_0:0 < \alpha (x) < k/2 \}$
and
$V^-= \{x \in V (G) - z_0: -k/2< \alpha (x) <0\}$.

Note that if $k/2\stackrel{k}{\equiv} p(x)-d_G(x)/2$,
then $\alpha(x)$ has two possible
% $\alpha $-
values, namely $k/2$ and $-k/2$.
%--------------------------------------------------------------------
%
\vspace{3mm}
\\
{\bf Claim 5 . $V (G) -z_0 = V^+$ or $V(G) - z_0 = V^-$.}

By Claim 4, we have $V^+ =\emptyset$ or $V^- =\emptyset$. 
So it suffices to prove that $|\alpha (x)|< k/2$ for any vertex $x$
other than $z_0$.
If $x \in V (G) - z_0$ such that $|\alpha(x)|= k/2$, then for any vertex $y$ distinct from $x$ and $z_0$, we can choose
$\alpha(x) = k/2$ or $\alpha (x) =-k/2$ such that $\alpha (x)\alpha(y) \le 0$ and get a contradiction to Claim~4.
$\square$
%--------------------------------------------------------------------
%
\vspace{3mm}
\\
{\bf Part II. Minimum members of $\mathcal{M}$.}

Now choose $(G,p, z_0)$ to be a member of $\mathcal{M}$ such that $|E(G)|$ is minimum.
Without loss of generality, assume that $V (G) - z_0 = V^+$. For if $V (G) -z_0 = V^-$, we reverse the
directions of all edges incident with $z_0$ and replace $p(x)$ by $d_G(x) - p(x)$ for each vertex $x$ (including $z_0$).
%with $\alpha(x)\neq 0$.
 Then the resulting graph with the modified mapping satisfies $V (G) - z_0 = V^+$ and is
also a minimum member of $\mathcal{M}$.

For each vertex $x \in V (G) - z_0$,
\begin{equation*}
{d_G(x) \ge (2k - 2) + 2\alpha(x) \text{ and } 0 < \alpha(x) < k/2.
}\end{equation*}
%--------------------------------------------------------------------
%
%\vspace{3mm}
%\\
{\bf Claim 6. $d_G(z_0) = k + p(z_0)$, and all edges incident with $z_0$ are directed away from~$z_0$.}

By Claim 3, $z_0$ has a neighbour $x$. 
By Claim 5, $0 < \alpha (x) < k/2 $.
If $xz_0$ is directed toward $z_0$, 
then we delete $xz_0$. 
By a
proof similar to that of Claim 4, 
the resulting graph with modified 
mapping $p'=p-\chi_x$
 satisfies 
the conditions of
the theorem.
 Since $|V (G')|=|V (G)|$ and $|E(G')| < |E(G)|$,
 and $(G,p, z_0)$ is a smallest member of $\mathcal{M}$,
the pre-orientation can be extended to a $p'$-orientation of $G'$
 and then to 
a $p$-orientation 
of $G$ which contradicts the fact that $(G,p, z_0)$ is a counterexample.
So all edges incident with $z_0$ are directed away from $z_0$, and 
$d_G(z_0)=d^+_G(z_0) \stackrel{k}{\equiv} p(z_0)$.
Now, we can assume that 
$d^+_G(z_0) =d_G(z_0) = sk+p(z_0)$, where $s\ge 0$.
By 
condition (i),
 we have $d_G(z_0) \le (2k - 2) +2|\alpha(z_0)| \le 3k - 2$ and so $s\le 2$.
In the case $s=2$, we derive that
 $p(z_0)/2< |\alpha(z_0)|$.
Since
 $$p(z_0)\stackrel{k}{\equiv}\alpha(z_0)+d_G(z_0)/2\stackrel{k}{\equiv} \alpha(z_0)+(2k+p(z_0))/2 \stackrel{k}{\equiv} \alpha(z_0)+p(z_0)/2,$$
we also derive that 
$\alpha(z_0) \stackrel{k}{\equiv} p(z_0)/2$ which is a contradiction.
By Claim 3, we have $d_G(z_0) = k + p(z_0)$. $\square$
%--------------------------------------------------------------------
%
\vspace{3mm}
\\
{\bf The final step: $(G ,p, z_0 )$ is not a counterexample.}

By Claim 6, 
let $x$ be a neighbour of $z_0$, 
and let $e$ be an edge directed from $z_0$ to $x$. 
We replace $e$ by
$k - 1$ multiple directed edges from $x$ to $z_0$.
Let $G'$ be the resulting graph with $p'= p-\chi_x-\chi_{z_0}$. 
We are going to prove that $G'$ with the mapping $p'$ satisfies all conditions of the theorem and, 
furthermore, 
$-k/2 < \alpha'(x) <0$ for the vertex $x$.
By Claim 6,
$d_G(z_0) = k + p(z_0)$.
Since $p'(z_0) = p(z_0)-1$ and $d_{G'}(z_0) = d_G(z_0) + k - 2$, we have 
$$\alpha' (z_0) \stackrel{k}{\equiv} p'(z_0)- d_{G'}(z_0)/2 \stackrel{k}{\equiv} p(z_0)-1-(d_G(z_0) + k - 2)/2 \stackrel{k}{\equiv}p(z_0)/2.$$
This implies that $2|\alpha' (z_0)| = p(z_0)$ 
and therefore, $d_{G'}(z_0) =(2k - 2) +2|\alpha' (z_0)|$. 
So, 
condition (i)
 is satisfied for $G'$
 and $p'$.

For
 condition (ii),
 we only need to consider $x$ and vertex sets containing $x$. 
 Since $|\alpha (x)| \ge 1/2$, we have 
$$d_{G'}(x) = d_G(x) +(k - 2) \ge (2k - 2) +2|\alpha (x)|\,+ (k - 2) \ge 3k - 3,$$
and hence $d_{G'}(x)\ge (2k - 2) +2|\alpha' (x)|$.
In addition,
$$\alpha'(x)\stackrel{k}{\equiv} p'(x)-d_{G'}(x) /2\stackrel{k}{\equiv} p(x)-1 - \big(d_{G}(x) + (k - 2)\big)/2\stackrel{k}{\equiv} \alpha(x) - k/2.$$
Since 
 $0 < \alpha (x)< k/2 $, we have $\alpha' (x) = \alpha(x) - k/2$ and so $-k/2< \alpha' (x) < 0$.
By Claim~1, for any non-trivial vertex set $A$ of $G$
 described in
 condition (ii)
 and
containing $x$, we also have 
$$d_{G'}(A) = d_G(A) + k - 2 \ge 3k - 2 = (2k - 2) + k \ge (2k - 2) +2|\alpha' (A)|.$$
So, 
condition (ii) 
is also satisfied.

Now if $(G',p', z_0)$ is also a counterexample,
 then $(G',p', z_0)\in \mathcal{M}$, 
since $|V (G')|=|V (G)|$ and $|E(G' - z_0)|=|E(G - z_0)|$. 
But we have $V'^+ = V (G')- \{z_0, x\}$ and $V'^- =\{x\}$, a contradiction to Claim 5.
So $(G',p', z_0)$ is not a counterexample, and hence $G'$
 has a $p'$-orientation. 
Then the corresponding orientation of $G$ (obtained by replacing the $k - 1$ edges from $x$ to $z_0$ with one edge in opposite
direction) is a $p$-orientation of $G$ satisfying the theorem.
This completes the proof.
}\end{proof}
When $z_0$ does not have small enough degree, one can replace the following version of Theorem~\ref{Orientation:3k-3:thm:alpha}.
Note that by ignoring the extra condition on $z_0$, the proof can easily be obtained after adding an additional vertex of degree zero which plays the role of the vertex $z_0$ in Theorem~\ref{Orientation:3k-3:thm:alpha}.
\begin{cor}\label{Orientation:3k-3:cor:edge:stronger:z0:3k-3}
{Let $G$ be a
graph, let $k$ be an integer, $k\ge 3$, 
and let $p:V(G)\rightarrow Z_k$ be a mapping with
$|E(G)|\stackrel{k}{\equiv} \sum_{v\in V(G)}p(v)$. 
If for every vertex set $A$ with $\emptyset \subsetneq A \subsetneq V(G)$,
$d_G(A)\ge 2k-2+2|\alpha(A)|,$
then $G$ has a $p$-orientation such that for each vertex~$v$,
$$|d^+_G(v)-d_G(v)/2| \le k-1+|\alpha(v)|.$$
Furthermore, for an arbitrary vertex $z$, $d^+_{G}(z)$ 
can be assigned to any plausible integer value in whose interval.
}\end{cor}
\begin{proof}
{The proof is by induction on $|V(G)|+|E(G)|$.
For $|V(G)|\le 2$, the proof is trivial. Hence we may assume that $|V(G)|\ge 3$.
If $d_G(z)= 2k-2+2|\alpha(z)|$, then the proof can easily be derived from 
Theorem~\ref{Orientation:3k-3:thm:alpha}.
So, suppose $d_G(z)\ge 2k+2|\alpha(z)|$. 
We claim that for any vertex set $A\subsetneq V(G)\setminus \{z\}$ with $|A|\ge 2$, we have
 $d_G(A)\ge 2k+2|\alpha (A)|$
 and so $d_G(A^c)\ge 2k+2|\alpha (A^c)|$.
For, if $d_G(A)< 2k+2|\alpha (A)|$,
 then by a proof similar to that of
Claim~1, we apply induction to $G/A$ and then we apply Theorem~\ref{Orientation:3k-3:thm:alpha} to $G/A^c$. 
If $z$ has at least two neighbours, then we lift one non-parallel pair of edges incident with $z$.
Otherwise, if $z$ has only one neighbour $y$, we lift one parallel pair of edges incident with $z$ and $y$. In this case, we have
$$d_G(y)\ge d_G(z)+d_G( \{z,y\})\ge 
 2k+2k-2\ge\; 2k+2|\alpha(y)|.$$
By applying the induction hypothesis the proof can be completed.
}\end{proof}
\begin{cor}\label{Orientation:3k-3:cor:edge}
{Let $G$ be a graph, let $k$ be a positive integer, let $p:V(G)\rightarrow Z_k$ be a mapping with 
$|E(G)|\stackrel{k}{\equiv}\sum_{v\in V(G)}p(v)$.
If $G$ is $(3k-3)$-edge-connected, then it has a $p$-orientation such that for each vertex $v$,
$$\lfloor \frac{d_G(v)}{2}\rfloor - (k-1)\le\; d^+_G(v) \; \le \lceil \frac{d_G(v)}{2}\rceil+(k-1).
$$
Furthermore, for an arbitrary vertex $z$, $d^+_{G}(z)$
can be assigned to any plausible integer value in whose interval.
}\end{cor}
\begin{proof}
{The proof of $k=1$ is clear; note that the extra condition on $z$ can be obtained by reversing the orientation (if necessary). 
The proof of $k=2$ follows from Theorem~\ref{Orientation:modulo2:thm}. The proof of $k=3$ follows from Corollary~\ref{Orientation:3k-3:cor:edge:stronger:z0:3k-3}.
Note that the condition
 $|d^+_{G}(v)-d_{G}(v)/2|\le k-1+|\alpha(v)|$ directly
 implies that $|d^+_{G}(v)-d_{G}(v)/2|< k$. Equivalently,
$\lfloor \frac{d_{G}(v)}{2}\rfloor - (k-1)\le d^+_{G}(v) \le \lceil \frac{d_{G}(v)}{2}\rceil+(k-1)$, 
because 
$d^+_{G}(v)-d_{G}(v)/2\in \{\pm |\alpha(v)|, \pm (k-|\alpha(v)|)\}$ and also $ d^+_{G}(v)-d_{G}(v)/2=0$ when $\alpha(v)=0$.
}\end{proof}
%
%%==============================================================================

\subsection{Replacing odd-edge-connectivity condition}
\label{sec:odd-edg}
Motivated by Theorem~4.12 in~\cite{Lovasz-Thomassen-Wu-Zhang-2013}, we improve Theorem~\ref{Orientation:3k-3:thm:alpha} as the following strengthened version which discounts the condition $d_G(A)\ge 2k-2+2|\alpha(A)|$ for any vertex set $A$ with $\alpha(A)= 0$.
\begin{thm}\label{Orientation:3k-3:thm:edge:stronger:odd-edge}
{Let $G$ be a
graph with $z_0\in V(G)$, let $k$ be an integer, $k\ge 3$, and let $p:V(G)\rightarrow Z_k$ be a mapping with
$|E(G)|\stackrel{k}{\equiv} \sum_{v\in V(G)}p(v)$. 
Let $D_{z_0}$ 
be a pre-orientation 
of $E(z_0)$ that is the set of edges incident with $z_0$.
%Let $V_0=\{v\in V(G)-z_0: \alpha(v)=0\}$.
% If $V_0\neq \emptyset$, we let $v_0$ be a vertex of $V_0$ with smallest degree. 
 Assume that 
\begin{enumerate}{

\item [{\upshape$(i)$}]  %1
%, when $Z_0=\{z_0\}$.
 $\alpha(z_0)\neq 0$.

\item [{\upshape$(ii)$}]  %1
%$|V(G)|\ge 3$.
$d_G(z_0)\le 2k-2+2|\alpha(z_0)|$, and the edges incident with $z_0$ are pre-directed such that $d^+_G(z_0)\stackrel{k}{\equiv}p(z_0)$.
%$|d^+_G(Z_0)-d_G(Z_0)/2| \le k-1+|\alpha(Z_0)|$.

%\item [{\upshape$(iii)$}]  %1
%$d_G(z_0)\le d_G(v)$ for each vertex $v$ with $\alpha(v)=0$ and $v\neq v_0$, if $\alpha(z_0)=0$.

\item [{\upshape$(iii)$}]  %1
$d_G(A)\ge 2k-2+2|\alpha(A)|$, for any vertex set $A$ with $\emptyset \subsetneq A \subsetneq V(G)\setminus \{z_0\}$ and $\alpha(A)\neq 0$.
}\end{enumerate}

Then the pre-orientation $D_{z_0}$ can be extended to a $p$-orientation $D$ of $G$ such that for each vertex~$v$,
$$|d^+_G(v)-d_G(v)/2| \le k-1+|\alpha(v)|.
$$
}\end{thm}
\begin{proof}
{The proof is by induction on $|V(G)|+|E(G)|$.
 For $|V(G)|\le 3$, the assertion holds by Theorem~\ref{Orientation:3k-3:thm:alpha}.
 So, suppose $|V(G)|\ge 4$.
%Assume (reductio ad absurdum) that $G$ is a smallest counterexample.
We claim that for any vertex set $A\subsetneq V(G)\setminus \{z_0\}$ such that $\alpha(A)\neq 0$ and $|A|\ge 2$, we have
 $d_G(A)\ge 2k+2|\alpha (A)|$.
For, if $d_G(A)< 2k+2|\alpha (A)|$,
 then by a proof similar to that of
Claim~1, we apply induction to $G/A$ and then to $G/A^c$. 
Then by a proof similar to that of Claim~2, we claim that there is no vertex $v_0$ of $G$ such that $\alpha(v_0)=0$; 
for otherwise we either remove $v_0$ or lift one pair of edges incident with $v_0$, and next we apply induction.
Now $G$ must have a vertex set $A\subsetneq V(G)\setminus \{z_0\}$ such that $\alpha(A)=0$, $|A|\ge 2$, and $d_G(A)\le 2k-2$.
For otherwise
$G$ satisfies the conditions of Theorem~\ref{Orientation:3k-3:thm:alpha}, and Theorem~\ref{Orientation:3k-3:thm:edge:stronger:odd-edge} follows.
Choose $A$ with minimal $|A|$.
We contract $A$ and use induction. 
Then we contract $A^c$ and by the
minimality of $A$ we can apply
Theorem~\ref{Orientation:3k-3:thm:alpha} to the graph $G/A^c$.
}\end{proof}

When $z_0$ does not have small enough degree, one can replace the following version of Theorem~\ref{Orientation:3k-3:thm:edge:stronger:odd-edge}.
\begin{cor}\label{Orientation:3k-3:cor:odd-edge:z0with-large-degree}
{Let $G$ be a
graph, let $k$ be an integer, $k\ge 3$, and let $p:V(G)\rightarrow Z_k$ be a mapping with
$|E(G)|\stackrel{k}{\equiv} \sum_{v\in V(G)}p(v)$. 
If for every vertex set $A$ with $\alpha(A)\neq 0$,
$d_G(A)\ge 2k-2+2|\alpha(A)|,$
then $G$ has a $p$-orientation such that for each vertex~$v$,
$$|d^+_G(v)-d_G(v)/2| \le k-1+|\alpha(v)|.$$
Furthermore, for an arbitrary vertex $z$, $d^+_{G}(z)$ 
can be assigned to any plausible integer value in whose interval.
}\end{cor}
\begin{proof}
{By induction on $|V(G)|+|E(G)|$.
For $|V(G)|\le 2$, the proof is trivial. Hence we may assume that $|V(G)|\ge 3$.
Note that if $\alpha=0$, then the graph $G$ whose degrees are even and the theorem clearly holds.
Also if $\alpha(z)=0$, then we can take another vertex as $z$ without this property.
Hence we may assume that $\alpha(z)\neq 0$.
If $d_G(z)= 2k-2+2|\alpha(z)|$, then the conclusion trivially holds, using
Theorem~\ref{Orientation:3k-3:thm:edge:stronger:odd-edge}.
So, suppose $d_G(z)\ge 2k+2|\alpha(z)|$. 
We claim that for any vertex set $A\subsetneq V(G)\setminus \{z\}$ such that $\alpha(A)\neq 0$ and $|A|\ge 2$, we have
 $d_G(A)\ge 2k+2|\alpha (A)|$
 and so $d_G(A^c)\ge 2k+2|\alpha (A^c)|$.
For, if $d_G(A)< 2k+2|\alpha (A)|$,
 then by a proof similar to that of
Claim~1, we apply induction to $G/A$ and then we apply Theorem~\ref{Orientation:3k-3:thm:edge:stronger:odd-edge} to $G/A^c$. 
If $z$ has at least two neighbours, then we lift one non-parallel pair of edges incident with $z$.
Otherwise, if $z$ has only one neighbour $y$, we lift one parallel pair of edges incident with $z$ and $y$. In this case, we have
$$d_G(y)\ge d_G(z)+d_G( \{z,y\})\ge 
 \begin{cases}
2k+2|\alpha(z)|=\; 2k+2|\alpha(y)|, &\text{if $\alpha(\{z,y\})=0$};\\
 2k+2k-2\ge\; 2k+2|\alpha(y)|, &\text{if $\alpha(\{z,y\})\neq 0$}.
 \end{cases}$$
By applying the induction hypothesis the proof can be completed.
}\end{proof}
\begin{cor}\label{Orientation:3k-3:cor:odd-edge}
{Let $G$ be a graph and let $k$ be an odd positive integer.
If $G$ is odd-$(3k-2)$-edge-connected, then it has an orientation such that for each vertex~$v$,
$$d^+_G(v)\in \{\frac{d_G(v)}{2}-\frac{k}{2},\; \frac{d_G(v)}{2},\; \frac{d_G(v)}{2}+\frac{k}{2}\}.$$
}\end{cor}
\begin{proof}
{The proof of $k=1$ is clear. 
So, suppose  $k\ge 3$.
For each vertex $v$ with even degree, define $p=0$ (mod $k$), and define $p=d_G(v)/2+k/2$ (mod $k$) otherwise.
By Corollary~\ref{Orientation:3k-3:cor:odd-edge:z0with-large-degree}, the graph $G$ has a  $p$-orientation such that for each vertex $v$, $|d^+_G(v)-d_G(v)/2|<k$. This can complete the proof.
}\end{proof}
\begin{cor}\label{Orientation:3k-3:cor:Eulerian-oddset}
{Let $G$ be a graph with even degrees, let $k$ be a positive integer, and let $Q\subseteq V(G)$ with $|Q|$ even.
If for every $A\subseteq V(G)$ with $|A\cap Q|$ odd, $d_G(A)\ge 6k-2,$
then $G$ has an orientation such that for each vertex $v$,
 $$d_G^+(v) \in \begin{cases}
\{\frac{d_G(v)}{2}\},	&\text{when $v\not \in Q$};\\
\{\frac{d_G(v)}{2}-k, \frac{d_G(v)}{2}+k\},	&\text{when $v\in Q$}.
\end {cases}$$
}\end{cor}
\begin{proof}
{For each $v\in Q$, define $p=d_G(v)/2+k$ (mod $2k$), and for each $v\in V(G)\setminus Q$, define $p=d_G(v)/2$ (mod $2k$).
It is easy to check that for every vertex set $A$, we have $\alpha_G(A,p)=0$ when $|A\cap Q|$ is even, and 
$|\alpha_G(A, p)|=k$ when $|A\cap Q|$ is odd.
Thus by Corollary~\ref{Orientation:3k-3:cor:odd-edge:z0with-large-degree}, the graph $G$ has a $p$-orientation modulo $2k$. 
For the special case $k=1$, we can replace the condition $d_G(A)\ge 2$ for every vertex set $A$ with $|A\cap Q|$ odd. 
For this purpose, we need to apply Theorem~\ref{Orientation:modulo2:thm} to each component of $G$ separately.
}\end{proof}
%
%
%==============================================================================
\subsection{A new presentation: a lower bound independent of degrees of vertices}
Our aim in this subsection is to introduce the following equivalent version of Theorem~\ref{Orientation:3k-3:thm:alpha} which is useful for working with essentially edge-connected graphs.
This version can directly be proved using the same arguments stated in the proof of 
Theorem~\ref{Orientation:3k-3:thm:alpha}. 
However, some parts in the proof would be shorter to state,
some parts need more extra efforts. 
\begin{thm}\label{thm:new-lower-bound}
{Let $G$ be a
graph with $z_0\in V(G)$, let $k$ be an integer, $k\ge 3$, and let $p:V(G)\rightarrow Z_k$ be a mapping with
$|E(G)|\stackrel{k}{\equiv} \sum_{v\in V(G)}p(v)$. Let $D_{z_0}$ 
be a pre-orientation 
of $E(z_0)$ that is the set of edges incident with $z_0$.
Let $V_0=\{v\in V(G)-z_0: p(v)\stackrel{k}{\equiv}d_G(v)/2\}$.
 If $V_0\neq \emptyset$, we let $v_0$ be a vertex of $V_0$ with smallest degree. Assume that 
\begin{enumerate}{

\item [{\upshape$(i)$}]  %1
 $d_G(z_0)\le 2k-1+p(z_0)$, and the edges incident with $z_0$ are pre-directed such that $d^+_G(z_0)\stackrel{k}{\equiv}p(z_0)$.

\item [{\upshape$(ii)$}]  %1
$d_G(A)\ge 2k-1+[p(A)]_k$, for any vertex set $A$ with $\emptyset \subsetneq A \subsetneq V(G)\setminus \{z_0\}$ and $A\neq \{v_0\}$.
}\end{enumerate}

Then the pre-orientation $D_{z_0}$ can be extended to a $p$-orientation $D$ of $G$ such that for each vertex~$v$,
$$|d^+_G(v)-d_G(v)/2| <k.$$
}\end{thm}
\begin{cor}\label{Orientation:3k-3:cor:essentially:edge}
{Let $G$ be a graph, let $k$ be an integer, $k\ge 3$, let $p:V(G)\rightarrow Z_k$ be a mapping with 
$|E(G)|\stackrel{k}{\equiv}\sum_{v\in V(G)}p(v)$.
If $G$ is essentially $(3k-3)$-edge-connected and for each vertex $v$, $d_G(v)\ge 2k-1+[p(v)]_k$, then $G$ has a $p$-orientation such that for each vertex $v$,
$$\lfloor \frac{d_G(v)}{2}\rfloor - (k-1)\le\; d^+_G(v) \; \le \lceil \frac{d_G(v)}{2}\rceil+(k-1).
$$
Furthermore, for an arbitrary vertex $z$, $d^+_{G}(z)$
can be assigned to any plausible integer value in whose interval.
}\end{cor}
\begin{cor}\label{cor:new-presentation:3k-3}
{Let $G$ be a graph, let $k$ be an integer, $k\ge 3$, and let $p:V(G)\rightarrow Z_k$ be a mapping with
$|E(G)| \stackrel{k}{\equiv} \sum_{v\in V(G)}p(v)$ such that for each vertex $v$, $p(v)=0$ or $p(v)\stackrel{k}{\equiv}d_G(v)$.
If $G$ is $(2k-1)$-edge-connected essentially $(3k-3)$-edge-connected, then it has a $p$-orientation.
}\end{cor}
The following proposition gives some useful relationship between the set functions $\alpha(A) $ and $[p(A)]_k$, which
is a useful tool to show that why Theorems~\ref{Orientation:3k-3:thm:alpha} and~\ref{thm:new-lower-bound} are equivalent.
\begin{proposition}\label{prop:new-presentation}
{Let $k$ be a positive integer and let $d$ and $p$ be two nonnegative integers with $p\in Z_k$.
Take $\alpha$ to be a rational number such that
$\alpha\in \{0,\pm 1/2,\ldots,\pm k/2\}$ and
$\alpha\stackrel{k}{\equiv}p-d/2$.
Assume that $2k-2\le d \le 3k-2$.
 Then following statements hold:
\begin{enumerate}{
\item \label{new-con:0} 
$d\le 2k-2+2|\alpha|$  if and only if $d\le 2k-1+p$.

\item \label{new-con:3} 
$d= 2k-2+2|\alpha|$ if and only if $d= 2k-1+p$ or $p=k-1$.

\item \label{new-con:1} 
$d\ge 2k-2+2|\alpha|$ if and only if $d\ge 2k-1+[p]_k$ .

\item \label{new-con:2} 
 $d\ge 2k-1+[p]_k$ if and only if $d\ge 2k-1+[d-p]_k$.
}\end{enumerate}
}\end{proposition}
\begin{proof}
{Let $a$ be the unique integer in $\{-1,\ldots, k-1\}$ such that $d=2k-1+a$.
It is not difficult to show that
 $$ \begin{cases}
2|\alpha|-1<a,	&\text{when $p<a$};\\
2|\alpha|-1=a,	&\text{when $p=a$ or $p=k-1$};\\
2|\alpha|-1>a,	&\text{when $a< p<k-1$}.
\end {cases}$$
This can confirm items (1), (2), and (3).
Now, assume that  $d\ge 2k-1+[p]_k$ so that $a\ge [p]_k$.
Thus $0\le p\le a$ or $p=k-1$ which implies that $a\ge [a-1-p]_k =[d-p]_k$.
This can complete the proof.
}\end{proof}
%
%==============================================================================
%
%
%
%
%
%
%
%%==============================================================================
%
\section{Orientations modulo $k$: partition-connected versions}
\label{sec:partition-connected}
In this section, we improve the needed edge-connectivity in Corollary~\ref{Orientation:3k-3:cor:edge}, but require the graph to have
 many edge-disjoint spanning trees.
%
%
%==============================================================================
%
%
%
%
%
%
%
%
%==============================================================================
%
%
\subsection{Basic tools: Lifting operations preserving partition-connectivity}
In this subsection, we present a sufficient condition for the existence of lifting operations which preserves
 tree-connectivity.
\begin{thm}\label{thm:preserving}
{Let $G$ be a general graph with $z\in V(G)$ and let $l_0$ be a nonnegative integer-valued function on $V(G)$.
Assume that $z$ is not incident with loops.
If $G$ contains an $(m, l_0)$-partition-connected factor $H$ with $d_G(z)\ge 2d_H(z)-2l_0(z)-2m$, then there are 
$d_H(z)-l_0(z)-m$ pair of edges incident with $z$ such that by lifting them the resulting general graph $G_0$ with
 $V(G_0)=V(G)\setminus \{z\}$ is still $(m,l_0)$-partition-connected.
}\end{thm}
\begin{proof}
{First assume that $m=0$. 
By the assumption, the graph $H$ has an orientation such for each vertex $v$, $d^+_H(v)\ge l_0(v)$.
Let $u_1z,\ldots, u_nz$ be the edges of $H$ incident with $z$ directed toward $z$. 
Since $d_G(z)\ge 2(d_H(z)-l_0(z))\ge 2d^-_H(z)$, there are
at least $n$ edges $v_1z,\ldots, v_nz$ of $E(G)\setminus \{u_1z,\ldots, u_nz\}$ incident with $z$.
Define $G_0$ to be the directed general graph with $V(G_0)=V(G)-z$ consisting of all directed edges of $G-z$ along with the new directed edges $u_1v_1,\ldots, v_nu_n$ which $u_iv_i$ is directed from $u_i$ to $v_i$. 
It is easy to check that $G_0$ is $(0,l_0)$-partition-connected.

Now, assume that $m= 1$ and $l_0=0$. 
We prove this case by induction on the number of components of $H-z$.
Let $C_1,\ldots, C_n$ be the components of $H-z$ and let 
 $u_1z,\ldots, u_nz$ be $n$ edges of $H$ incident with $z$ such that $u_i\in V(C_i)$.
If $n=1$, then $G-z$ is connected and the proof is clear.
If $n=2$, define $G_0$ to be the graph with $V(G_0)=V(G)-z$ consisting of all edges of $G-z$ along with the new edge
$u_1u_2$. It is easy to check that $G_0$ is connected.
So, suppose $n\ge 3$.
Since $d_G(z)> d_H(z)$, there is an edge $v_1z$ of $E(G)\setminus E(H)$ incident with $z$.
We may assume that $v_1\not \in V(C_1)$.
Define $G'$ and $H'$ to be the graphs obtained from $G$ and $H$ by removing the edges $u_1z$ and $v_1z$ and adding the new edge $u_1v_1$.
 It is easy to check that $H'$ is still connected and $H'-z$ has $n-1$ components. 
Moreover, we have $d_{G'}(z)=d_G(z)-2$ and $d_{H'}(z)=d_H(z)-1$, which implies that $d_{G'}(z)\ge 2d_{H'}(z)-2$.
Now, by applying induction to the graphs $G'$ and $H'$ the proof of this part can easily be completed. 

Now, we are going to prove the remaining cases by induction on $m$.
Suppose $m\ge 1$.
Set $m_1=1$, $l_1=0$, $m_2=m-1$, and $l_2=l_0$.
Decompose $H$ into two factors $H_1$ and $H_2$ such that each $H_i$ is $(l_i+m_i)$-partition-connected.
Also, decompose $G$ into two factors $G_1$ and $G_2$ such that each $G_i$ contains $H_i$.
Since $d_G(z)\ge 2d_H(z)-2l_0(z)-2m$, we must have $d_{G_i}(z)\ge 2d_{H_i}(z)-2l_i(z)-2m_i$ for at least an integer $i\in \{1,2\}$.
Let $j\in \{1,2\}$ with $j\neq i$.
 By induction hypothesis, there are $d_{H_i}(z)-l_i(z)-m_i$ pair of edges of $G_i$ incident with $z$ such that by lifting them the resulting general graph $G'_i$ with $V(G'_i)=V(G_i)\setminus \{z\}$ is still $(m_i, l_i)$-partition-connected.
Define $\mathcal{G}_j$ to be the factor of $G$ consisting of all edges of $G_j$ together with all remaining edges of $G_i$ incident with $z$ which are not lifted.
According to the construction of $\mathcal{G}_j$, we must have 
$d_{\mathcal{G}_j}(z)= d_G(z)-2 (d_{H_i}(z)-l_i(z)-m_i)\ge 2d_{H_j}(z)-2l_j(z)-2m_j$.
By induction hypothesis, there are $d_{H_j}(z)-l_j(z)-m_j$ pair of edges of $\mathcal{G}_j$ incident with $z$ such that by lifting them the resulting general graph $G'_j$ with $V(G'_j)=V(\mathcal{G}_j)\setminus \{z\}$ is still $(m_j, l_j)$-partition-connected.
It is enough, now, to define $G_0=G'_1\cup G'_2$. 
Hence the theorem is proved.
}\end{proof}
\begin{cor}{\rm (\cite{Fekete-Szego-2007})}\label{PreservingTree-Connectivity:cor}
{Let $G$ be an $m$-tree-connected graph with $z\in V(G)$. If $d_G(z)\le 2m$, 
then there are 
$d_G(z)-m$
non-parallel pair of edges incident with $z$ such that by lifting them the resulting graph $G_0$ with $V(G_0)=V(G)\setminus \{z\}$ is still $m$-tree-connected.
}\end{cor}

\subsection{Graphs with partition-connectivity at least $(2k-2,l_0)$}
In this subsection, we improve the needed edge-connectivity in Corollary~\ref{Orientation:3k-3:cor:edge}, but require the graph to have
 many edge-disjoint spanning trees.
\begin{thm}\label{Tree:thm:(k-1)-bound}
{Let $G$ be a general graph with $z\in V(G)$, let $k$ be an integer, $k\ge 3$, and 
 let $p:V(G)\rightarrow Z_k$ be a mapping with $|E(G)|\stackrel{k}{\equiv}\sum_{v\in V(G)}p(v)$.
Let $s$, $s_0$, and $l_0$ be three integer-valued functions on $V(G)$ satisfying $s+s_0+k-1\le d_G$ and
$\max\{s,s_0\}\le l_0+(k-1)(1-\chi_z)$. 
If $G$ is $(2k-2, l_0)$-partition-connected, then it has a $p$-orientation such that for each vertex $v$, 
$$s(v)\le d_{G}^+(v)\le d_G(v)-s_0(v).$$
}\end{thm}
\begin{proof}
{We may assume that $l_0$ is nonnegative and $G$ is loopless. 
The proof is by induction on $|V(G)|$.
For $|V(G)|\le 2$ the proof is straightforward.
So, suppose $|V(G)|\ge 3$.
For notational simplicity, let us define $m=2k-2$.
For proving the theorem, we shall consider the following four cases.
%======= Case 1
%
\vspace{3mm}
\\
{\bf Cases 1. There is a vertex $u\in V(G)\setminus \{z\}$ with $d_G(u)=2l_0(u)+2m-r$ such that 
$0\le r\le l_0(u)+m$ and 
 $p(u)\stackrel{k}{\equiv}l_0(u)+m-i$ where $0 \le i \le \min \{r,k-1\}$.}

By Theorem~\ref{thm:preserving},
there are $l_0(u)+m-r$ pair of edges incident with $u$ 
such that by lifting them the resulting general graph $H$ with $V(H)=V(G)\setminus \{u\}$
is still $(m, l_0)$-partition-connected.
Obviously, $d_R(u)=d_G(u)-2(l_0(u)+m-r)=r$,
where $R$ is the factor of $G$ consisting of all edges incident with $u$ that are not lifted.
Since $i\le r$, the edges of $R$ can be orientated
such that $ d^+_R(u) \stackrel{k}{\equiv}r-i$. 
Define $s'(u)=s(u)-(l_0(u)+m-r)$ and $s'_0(u)=s_0(u)-(l_0(u)+m-r)$.
By the assumption, we must have $\max\{s'(u),s'_0(u)\}\le d_R(u)-(k-1)$,
 and $s'(u)+s'_0(u)\le d_R(u)-(k-1)$.
Therefore, if $d_R(u)\ge k-1$ then the orientation of $R$ can be selected such that
$s'(u)\le d^+_R(u)\le d_R(u)-s'_0(u)$.
If $d_R(u)\le k-1$, then $\max\{s(v),s_0(v)\}\le 0$ and so we must automatically have
$$s'(u)\le 0\le d^+_R(u)\le d_R(u) \le d_R(u)-s'_0(u).$$
Now, for each vertex $v$ of $H$, define $p'(v)=p(v)-d^+_R(v)$.
It is easy to check that $|E(H)|\stackrel{k}{\equiv}\sum_{v\in V(H)}p'(v)$.
Obviously, $\max\{s(v)-d_{R}^+(v), s_0(v)-d_{R}^-(v) \}\le \max\{s(v),s_0(v)\}\le l_0(v)+k-1$ and 
$s(v)-d_{R}^+(v)+s_0(v)-d_{R}^-(v) +k-1 \le d_{H}(v)$. 
Thus by the induction hypothesis, 
$H$ has a $p'$-orientation modulo $k$ such that for each $v\in V(H)$, 
$$s(v)-d_{R}^+(v)\le d_{H}^+(v)\le d_{H}(v)-(s_0(v)-d_{R}^-(v))=d_{G}(v)-s_0(v)-d^+_{R}(v).$$
This orientation induces a $p$-orientation for $G$ such that for each $v\in V(H)$, 
 $d^+_G(v)= d^+_{H}(v)+d^+_R(v)$, and also $d_{G}^+(u)= d^+_R(u)+l_0(u)+m-r.$
This can complete the proof of Case 1. $\square$
%
%
%======= Case 2
%
\vspace{3mm}
\\
{\bf Case 2. $d_G(z)<2l_0(z)+k-1$.}

Since $d_G(z)\ge 2k-2$, we must have $l_0(z)>0$ and hence there is an edge $zu$ incident with $z$
 such that the graph $G_0$ is $(m, l_0-\chi_z)$-partition-connected, where $G_0=G-zu$. 

First assume that $s(z)<l_0(z)$ which implies that $s(z)\le l_0(z)-\chi_z(z)$.
Thus by the induction hypothesis, the graph $G_0$ has a $(p-\chi_u)$-orientation such that for each vertex $v$,
$s(v)-\chi_u(v)\le d_{G_0}^+(v)\le d_{G_0}(v)-(s_0(v)-\chi_z(v))$.
Now, this orientation induces the desired $p$-orientation for $G$ by adding an edge directed from $u$ to $z$.

Now, assume that $s(z)=l_0(z)$.
This implies that $s_0(z)<l_0(z)$, because $s(z)+s_0(z)+k-1\le d_G(z)<2l_0(z)+k-1$.
Thus by the induction hypothesis, the graph $G_0$ has a $(p-\chi_z)$-orientation such that for each vertex $v$,
$s(v)-\chi_z(v)\le d_{G_0}^+(v)\le d_{G_0}(v)-(s_0(v)-\chi_u(v))$.
Now, this orientation induces the desired $p$-orientation for $G$ by adding an edge directed from $z$ to $u$.
This completes the proof of Case 2. $\square$
%
%======= Case 3
%
\vspace{3mm}
\\
{\bf Case 3. There is a nonempty proper subset $A$ of $V(G)\setminus \{z\}$ such that 
 $d_G(A) < 2k-2+2|\alpha(A)|$.}

By the first case, we must have $|A|\ge 2$, because $d_G(v)\ge 2l_0(v)+3k-3$ for each $v\in V(G)\setminus \{z\}$.
Choose $A$ with minimal $|A|$.
We contract $A$ and use induction.
Note that $G/A$ is also $(m, l_0)$-partition-connected, where $l_0(A)=0$.
Then we contract $A^c$ and by the
minimality of $A$, we can apply
Theorem~\ref{Orientation:3k-3:thm:alpha} to the graph $G/A^c$.
 For verifying the condition on out-degrees of vertices of $A$, we can apply the same arguments stated in the next case.
$\square$
%
%======= Case 4
%
\vspace{3mm}
\\
{\bf Case 4: For every nonempty proper subset $A$ of $V(G)\setminus \{z\}$, $d_G(A) \ge 2k-2+2|\alpha(A)|$.}

By applying Theorem~\ref{Orientation:3k-3:thm:alpha} or Corollary~\ref{Orientation:3k-3:cor:edge:stronger:z0:3k-3} (with respect to the case that $d_G(z)\le 2k-2+2|\alpha(z)|$ or not), 
the graph $G$ has a $p$-orientation 
such that $|d^+_G(z)-d_G(z)/2|\le k/2$ and 
$|d^+_G(v)-d_G(v)/2|<k$ for all vertices~$v$.
According to Case 2, $d_G(z)\ge 2l_0(z)+k-1$
which implies that $s_0(z)\le l_0(z)\le d^+_G(z) \le d_G(z)-l_0(z)\le d_G(z)-s_0(z)$.
Let $v\in V(G)\setminus \{z\}$.
If $d_G(v)\ge 2l_0(v)+2m$, then we must automatically have 
$ s(v)\le l_0(v)+k-1\le d^+_G(v) \le d_G(v)-l_0(v)-(k-1)\le d_G(v)-s_0(v).$
Otherwise, $d_G(v)=2l_0(v)+2m-r$ in which $0< r< k-1$ and 
$ l_0(v)+m-k- r/2< d^+_G(v) < l_0(v)+m+k- r/2$.
According to Case 1, 
$$ d^+_G(v)\notin \{l_0(v)+m-k-i: 0\le i \le r\}\, \cup\, \{l_0(v)+m+k-i:0\le i \le r\},$$
which again implies that
$s(v)\le  l_0(v)+m-(k-1)\le d^+_G(v) \le d_G(v)-l_0(v)-m+(k-1)\le d_G(v)-s_0(v).$
Hence the proof is completed.
}\end{proof}
In the following theorem, we shall restate a simpler version of Theorem~\ref{Tree:thm:(k-1)-bound}
 which is refined by involving extension of pre-orientations.
\begin{thm}\label{thm:pre-orientation:tree-connected}
{Let $G$ be a non-trivial graph, let $k$ be an integer, $k\ge 3$, and let $p:V(G)\rightarrow Z_k$ be a mapping with  $|E(G)|\stackrel{k}{\equiv}\sum_{v\in V(G)}p(v)$. If $G$ is $(2k-2)$-tree-connected, then it has a $p$-orientation such that for each vertex~$v$, $$k/2-1\le d^+_G(v)\le d_G(v)-k/2+1.$$
}\end{thm}
\begin{proof}
{The proof can be obtained by induction on $|V(G)|$ as the arguments stated in the proof Theorem~\ref{Tree:thm:(k-1)-bound}.
For proving the assertion, one can consider the following three cases:  Case (i) there is a vertex $u$ with $d_G(u)\le  3k-3$. In this case, we should apply Theorem~\ref{PreservingTree-Connectivity:cor}. Case (ii) there is a vertex set $A$ satisfying $2\le |A|\le |V(G)|-2$ and $d_G(A)\le  2k-2+2|\alpha(A)|$. In this case,  we choose $A$ with minimal $|A|$. We contract $A$ and use induction. Note that $G/A$ is also $(2k-2)$-tree-connected. Then we contract $A^c$ and by the minimality of $A$, we can apply Theorem~\ref{Orientation:3k-3:thm:alpha} to the graph $G/A^c$. Case (iii)  fore every nonempty proper subset $A$ of $V(G)$, $d_G(A)\ge   2k-2+2|\alpha(A)|$.  In this case, we should apply Corollary~\ref{Orientation:3k-3:cor:edge:stronger:z0:3k-3}.
}\end{proof}
\begin{cor}\label{cor:positive-orientation}
{Let $G$ be a non-trivial graph, let $k$ be an integer, $k\ge 3$, and 
 let $p:V(G)\rightarrow Z_k$ be a mapping with 
$|E(G)|\stackrel{k}{\equiv}\sum_{v\in V(G)}p(v)$.
If $G$ is $(2k-2)$-tree-connected, then it has a $p$-orientation such that for each vertex $v$,
$0< d^+_G(v)< d_G(v).$
}\end{cor}
\begin{proof}
{Apply Theorem~\ref{thm:pre-orientation:tree-connected}.
}\end{proof}
%
%
%
%==============================================================================
%
%
%
%
%
%
%==============================================================================
%
\subsection{A generalization: definition of $\lambda_k$}
In this subsection, we shall define a parameter $\lambda_k$ and give an application of it on tree-connected graphs.
For this purpose, we need to form the following consequence of Theorem~\ref{Orientation:3k-3:thm:alpha} to show that the definition is well-defined by giving an upper bound on it. The following fact can be considered as an extension of Lemma 2.2 (i) in~\cite{Han-Lai-Li-2018}.
\begin{thm}
{Let $G$ be a graph with $z_0\in V(G)$, let $k$ be an integer, $k\ge 3$, and 
 let $p:V(G)\rightarrow Z_k$ be a mapping with 
$|E(G)|\stackrel{k}{\equiv}\sum_{v\in V(G)}p(v)$.
Let $D$ be a given pre-orientation of the edges incident with $z_0$ such that $d^+_{D}(z_0) \stackrel{k}{\equiv}p(z_0)$. 
If $d_G(z_0)\le 3k-2$ and 
 $d_G(A)\ge 3k-3$ for every nonempty proper subset $A$ of $V(G)\setminus \{z_0\}$,
then pre-orientation of $D$ can be extended a $p$-orientation of
 $G$.
}\end{thm}
\begin{proof}
{Let $D'$ be another pre-orientation of the edges incident with $z_0$ such that $d^+_{D'}(z_0)\stackrel{k}{\equiv}k-1$.
For each vertex $v$, we define $p'(v)=p(v)-d^+_{D}(v)+d^+_{D'}(v)$ (mod $k$).
Since $\sum_{v\in V(G)}d^+_{D}(v)\stackrel{k}{\equiv} d_G(z_0)\stackrel{k}{\equiv}\sum_{v\in V(G)}d^+_{D'}(v) $, 
it is easy to check that $|E(G)|\stackrel{k}{\equiv}\sum_{v\in V(G)}p(v)\stackrel{k}{\equiv}\sum_{v\in V(G)}p'(v)$.
Since $d_G(z_0)\le 3k-2=2k-1+p'(z_0)$, by Theorem~\ref{thm:new-lower-bound}, 
the pre-orientation of $D'$ can be extended to a $p'$-orientation modulo $k$ of $G$. 
Now, it is enough to replace the orientation of the edges of $D$ in the current orientation of $G$ to obtain the desired $p$-orientation.
}\end{proof}
For any positive integer $k$ with $k\ge 3$, we define $\lambda_k$ to be the smallest positive integer such that the following holds:
If $G$ is a graph with $z_0\in V(G)$ satisfying $d_G(z_0)< \lambda_k$ and 
 $d_G(A)\ge \lambda_k$ for every nonempty proper subset $A$ of $V(G)\setminus \{z_0\}$,
then any suitable pre-orientation of the edges incident with $z_0$ can be extended a $p$-orientation, where $p:V(G)\rightarrow Z_k$ is a given arbitrary mapping with
$|E(G)|\stackrel{k}{\equiv}\sum_{v\in V(G)}p(v)$.
According the definition of $\lambda_k$, one can now formulate the following result on modulo orientation of tree-connected graphs.
Note that for defining $\lambda_k$, we could also restrict our attention to $\lceil (\lambda_k+k-2)/2\rceil$-tree-connected graphs to deduce the following result alternatively.
\begin{thm}\label{Tree:thm:main}
{Let $G$ be a graph, let $k$ be an integer, $k\ge 3$, and 
 let $p:V(G)\rightarrow Z_k$ be a mapping with 
$|E(G)|\stackrel{k}{\equiv}\sum_{v\in V(G)}p(v)$.
If $G$ is $ \lceil (\lambda_k+k-2)/2\rceil$-tree-connected, then it has a $p$-orientation.
}\end{thm}
\begin{proof}
{Let $m= \lceil (\lambda_k+k-2)/2\rceil$. 
For proving the theorem, we consider the following three cases:
Case (i) there is a vertex $u$ with $d_G(u)< \lambda_k$.
In this case, by Corollary~\ref{PreservingTree-Connectivity:cor},
 there are 
$d_G(u)-m$
non-parallel pair of edges incident with $u$ such that by lifting them the resulting graph $G_0$ with $V(G_0)=V(G)\setminus \{u\}$ is still $m$-tree-connected.
Thus $2m-d_G(u)$ edges are not lifted.
Since $d_G(u)\le \lambda_k-1$, we must have $2m-d_G(u)\ge 2m -\lambda_k+1\ge k-1$.
Let $M$ be the factor of $G$ consisting of the edges incident with $u$ that are not lifted.
Since $|E(M)|\ge k-1$, the edges of $M$ can be directed such that $d^+_M(u)+(d_G(u)-d_M(u))/2\stackrel{k}{\equiv}p(u)$.
Now, by applying the induction hypothesis on $G_0$, the orientation of $M$ can be extended to a $p$-orientation of $G$.
Case (ii) there is a vertex set $A$ with $2\le |A|\le |V(G)|-2$ and $d_G(A)< \lambda_k$.
Choose $A$ with minimal $|A|$.
We contract $A$ and use induction.
Then we contract $A^c$ and by minimality of $A$ and the definition of $\lambda_k$, 
the pre-orientation can be extended to a $p$-orientation of $G$.
Note that $G/A$ is also $m$-tree-connected.
Case (iii) $G$ is $\lambda_k$-edge-connected. This case also follows from the definition of $\lambda_k$ by adding an artificial vertex $z_0$.
}\end{proof}
In 2012 Bar{\'a}t, Gerbner, and Thomass{\'e} proposed a conjecture on star-decomposition of simple graphs 
which can be reformulated to the following modulo orientation version, see~\cite{ManyEdgeDisjoint}. Recently, Han, Li, Wu, and Zhang~\cite{Han-Li-Wu-Zhang-2018} showed that the following conjecture cannot be developed to the class of $(k-1)$-tree-connected $(2k-1)$-edge-connected graphs, when $k\ge 11$.
By the above-mentioned theorem, one can confirm this conjecture for $ \lceil (3k-2)/2\rceil$-tree-connected graphs,
if $\lambda_k \le 2k$.
\begin{conj}{\rm(\cite{Barat-Gerbner-2014})}
{Let $G$ be a graph, let $k$ be an integer, $k\ge 3$, and 
 let $p:V(G)\rightarrow Z_k$ be a mapping with 
$|E(G)|\stackrel{k}{\equiv}\sum_{v\in V(G)}p(v)$.
If $G$ is $k$-tree-connected, then it has a $p$-orientation.
}\end{conj}
\begin{remark}
{Recently, Esperet, de Joannis de Verclos, Le, and Thomass{\'e} (2018)~\cite{Esperet-Verclos-Le-Thomasse-2018} 
 utilized Theorems~\ref{Into:thm:3k-2} to establish a result on additive bases and a result on weighted orientations.
By reviewing their proofs, we find out one can replace Theorem~\ref{Tree:thm:(k-1)-bound} in their proofs to get further improvements.
}\end{remark}
%
%==============================================================================
%
%
%
%
%
%
%
%
%
%
%								Acknowledgement
%==============================================================================	
%

%\bibliographystyle{siam}
%\bibliography{ref}
\end{document}